\documentclass[reqno,11pt]{amsart}

\usepackage{xcolor}
\usepackage{graphicx}
\usepackage[hidelinks]{hyperref}

\usepackage{latexsym,array}
\usepackage{amsfonts}
\usepackage{shadow}
\usepackage{amsbsy}
\usepackage{amssymb}
\usepackage{dsfont}
\usepackage{mathtools}
\usepackage{enumitem}
\usepackage{doi}

\usepackage[notref, notcite, final]{showkeys}
\usepackage{fullpage}
\usepackage{comment}
\usepackage[latin1]{inputenc}

\usepackage{cleveref}

\numberwithin{equation}{section}

\newcommand{\id}{\mathcal{I}}

\newcommand{\clos}[1]{\overline{#1}}

\newcommand{\dom}{\mathcal{D}}

\theoremstyle{theorem}
\newtheorem{theorem}{Theorem}[section]

\theoremstyle{definition}
\newtheorem{remark}[theorem]{{\bf Remark}}
\newtheorem{definition}[theorem]{Definition}
\newtheorem{problem}[theorem]{Problem}

\newcommand{\cc}{\mathbb{C}}

\newcommand{\hh}{\mathbb{H}}

\newcommand{\rr}{\mathbb{R}}
\renewcommand{\SS}{\mathbb{S}}

\newcommand{\boundOP}{\mathcal{B}}
\newcommand{\closOP}{\mathcal{K}}

\newcommand{\VEC}{\operatorname{Vec}}

\newcommand{\vx}{{{x}}}

\newcommand{\Q}{\mathcal{Q}}

\renewcommand{\Re}{\mathrm{Re}}

\newcommand{\uI}{j}

\crefname{enumi}{}{}
\crefname{enumii}{}{}

\title[ Fractional powers of vector operators and fractional Fourier's law]
{Fractional powers of vector operators and fractional Fourier's law in a Hilbert space}

\author[F. Colombo]{Fabrizio Colombo}
\address{(FC)
Politecnico di Milano\\Dipartimento di Matematica\\Via E. Bonardi, 9\\20133
Milano, Italy}
\email{fabrizio.colombo@polimi.it}
\author[J. Gantner]{Jonathan Gantner}
\address{(JG)
Politecnico di Milano\\Dipartimento di Matematica\\Via E. Bonardi, 9\\20133
Milano, Italy (On leave)
} \email{gantner.jonathan@gmail.com}

\begin{document}

\begin{abstract}
In this paper we give a concrete application of the spectral theory based on the notion of $S$-spectrum to fractional diffusion process.
Precisely, we consider the Fourier law for the propagation of the heat in non homogeneous materials,
that is the heat flow is given by the vector operator:
$$
T=e_1\,a(x)\partial_{x_1} +  e_2\,b(x)\partial_{x_2} +  e_3\,c(x)\partial_{x_3}
$$
where $e_\ell$, $\ell=1,2,3$ are orthogonal unit vectors in $\mathbb{R}^3$,
 $a$, $b$, $c$ are given real valued functions that depend on the space variables $x=(x_1,x_2,x_3)$, and possibly also on time.
 Using the $H^\infty$-version of the $S$-functional calculus we have recently defined fractional powers of quaternionic operators, which contain, as a
particular case, the vector operator $T$. Hence, we can define the non-local version $T^\alpha$, for $\alpha\in (0,1)$, of the Fourier law defined by $T$. We will see in this paper how we have to interpret $T^\alpha$,
when we introduce our new approach called: ``The $S$-spectrum approach to fractional diffusion processes''.
This new method allows us to enlarge the class of fractional diffusion and fractional evolution problems that can be defined and studied using our spectral theory based on the $S$-spectrum for vector operators.
 This paper is devoted to researchers working in fractional diffusion and fractional evolution problems, partial differential equations and non commutative operator theory.
  Our theory applies not only to the heat diffusion process but also to Fick's law and more in general it allows to compute the fractional powers of vector operators that arise in different fields of science and technology.
\end{abstract}
\maketitle
\vskip 1cm

\section{Introduction}

In a series of papers the authors and their collaborators have developed a new spectral theory based on the notions of $S$-spectrum that
applies to quaternionic operators and to $n$-tuples of not necessarily commuting operators. Such theory in
particular applies to vector operators  such as
\begin{equation}\label{FOUT}
T=e_1\,a(x)\partial_{x_1} +  e_2\,b(x)\partial_{x_2} +  e_3\,c(x)\partial_{x_3}
\end{equation}
that can represent Fourier's laws for non homogeneous materials in a bounded or unbounded domain.
Here $e_\ell$, for $\ell=1,2,3$, are orthogonal unit vectors,
 $a$, $b$, $c$ are given real valued functions that depend on the space variables $x=(x_1,x_2,x_3)$ and can depend on a parameter $t$ that represents the time in a given evolution process.
The development of the quaternionic spectral theory based on the $S$-spectrum started in 2016 with the discovery of the $S$-spectrum,
see the historical note in the book \cite{CGKBOOK}.
With the development of  this new branch of operator theory it was possible to extend  several classical results of complex operator theory
to quaternionic linear operators and in particular to linear vector operators.
An important application of the classical Riesz-Dunford functional calculus and of its extensions to unbounded operators,
like McIntosh's $H^\infty$-functional calculus, is the definition of fractional powers of sectorial operators.

This can also be done for quaternionic sectorial  operators (and in particular for sectorial vector operators) using the quaternionic version of the
$H^\infty$-functional calculus.
The study of the fractional powers of quaternionic and of vector operators started with the paper \cite{FJTAMS}.
The  $H^\infty$-functional calculus
 has been extended to the quaternionic setting in \cite{Hinfty},
following the original paper of McIntosh  \cite{McI1}.
This calculus can be developed  with less  restrictions on the operator as in \cite{Haase}. Recently, in  \cite{64FRAC},
 we have extended the  $H^\infty$-functional calculus in its full generality following this approach. In the same paper,  we also proved some non trivial results like the spectral mapping theorem in this setting and we used this functional calculus to define and study fractional powers of quaternionic linear sectorial operators and in particular vector operators.
 In our papers we have extended several classical results of fractional powers of operators in \cite{Balakrishnan,Guzman1,Kato,Komatsu1,McI1,Watanabe,Yosida} to the quaternionic setting.
\\
\\
This paper has two main goals. The first one is to define a new approach to fractional diffusion processes
 that is based on the spectral theory on the $S$-spectrum,
that we will call:
{\em  The $S$-spectrum approach to fractional diffusion problems}.
 The second is to give a concrete application to fractional diffusion problems and fractional evolution
when the operator $T$ defined in (\ref{FOUT}) has commuting components.
Our theory  applies more in general to operators $T$ with non commuting components ($a(x)\partial_{x_1}$, $b(x)\partial_{x_2}$, $c(x)\partial_{x_3}$)
but, for the sake of simplicity, we consider the commutative case. The strategy
is the same in this case, but the computations are a bit reduced, so that we can better explain how this theory works.
We will solve the following problem under suitable conditions.
\begin{problem}\label{PROB}
Let $\Omega \subseteq\mathbb{R}^3$ be a bounded or an unbounded domain.
Suppose that the heat flux $\mathbf{q}(x)$ in the non homogeneous material contained in $\Omega$ is given by the local operator (\ref{FOUT}).
Determine the non local law associated with $T$ by defining the fractional powers of $T$.
\end{problem}
\begin{remark}{\rm
We point out that the solution of Problem \ref{PROB} has several advantages as we will mention in the sequel. From the physical point of view, the most important one is that it allows to determine the fractional heat equation for non homogeneous materials by modifying only Fourier's law, but
 without modifying of the principle of conservation of energy.
}
\end{remark}
The motivation for solving the above problem is that the classical heat equation has solutions with unphysical properties. In particular heat can propagate with infinite speed. For this reason in the last years some non-local variations of the model
 have been introduced by replacing the Laplace operator $\Delta$ by the fractional powers $\Delta^\alpha$, for $\alpha\in(0,1)$.
  The operator $\Delta^\alpha$  is non-local and it takes non local phenomena in the heat propagation into account, see
for example \cite{BocurValdinoci,CafSil,CSV,CV2,GMPunzo,Vazquez} and the literature therein.
Before we state our main results  we need to recall some definitions.
\\
\\
The fundamental concepts of our theory are the $S$-resolvent set, the $S$-spectrum and the $S$-resolvent operators.
We denote the set of quaternions by $\mathbb{H}$ and consider a two-sided quaternionic Banach space $V$. We denote the set of closed quaternionic right linear operators on $V$ by  $\closOP(V)$ and we denote the Banach space of all bounded right linear operators on $V$ by $\mathcal{B}(V)$, which is endowed with the natural operator norm. For $T\in\closOP(V)$, we define the operator associated with the $S$-spectrum as:
\[ \Q_{s}(T) := T^2 - 2\Re(s)T + |s|^2\id, \qquad \text{for $s\in\hh$} \]
where $\Q_{s}(T):\mathcal{D}(T^2)\to V$,
and $\Re(s)$ denotes the real part and $|s|$ the modulus of the quaternion~$s$.
 We define the $S$-resolvent set of  $T$ as
\[\rho_S(T):= \{ s\in\hh: \Q_{s}(T)^{-1}\in\boundOP(V)\}\]
and the $S$-spectrum of $T$ as
\[\sigma_S(T):=\hh\setminus\rho_S(T).\]
 For $s\in\rho_S(T)$, the left $S$-resolvent operator is defined as
\begin{equation}
S_L^{-1}(s,T):= \Q_s(T)^{-1}\overline{s} -T\Q_s(T)^{-1}
\end{equation}
and the right $S$-resolvent operator is defined as
\begin{equation}
S_R^{-1}(s,T):=-(T-\id \overline{s})\Q_s(T)^{-1}.
\end{equation}

We will apply a modified version of  the quaternionic Balakrishnan formula
to the operator
\begin{equation}\label{TCOM}
T=e_1a_1(x_1)\partial_{x_1} + e_2a_2(x_2)\partial_{x_2} + e_3a_3(x_3)\partial_{x_3}.
\end{equation}
In order to determine fractional powers of $T$ we assume suitable conditions on the coefficients of the operator $T$ that allow us to show that the  $S$-resolvent operators $S_L^{-1}(s,T)$ and $S_R^{-1}(s,T)$ satisfy suitable estimates.
We choose now any $j\in \mathbb{H}$ such that  $j^2=-1$ and  for $\alpha\in(0,1)$ and  $v\in\dom(T)$ we set
\begin{equation}\label{BALA1}
P_{\alpha}(T)v := \frac{1}{2\pi} \int_{-j\rr}   S_L^{-1}(s,T)\,ds_j\, s^{\alpha-1} T v
\end{equation}
or
\begin{equation}\label{BALA2}
P_{\alpha}(T)v := \frac{1}{2\pi} \int_{-j\rr} s^{\alpha-1} \,ds_j\,  S_R^{-1}(s,T) T v
\end{equation}
where $ds_j=ds/j$. These integrals do not depend on the imaginary unit $j$ and both integrals define the same operator. Moreover, they correspond to a modified version of Balakrishnan's formula that takes only spectral points with positive real part into account.  These modifications are necessary because $s\mapsto s^\alpha$, for $\alpha\in (0,1)$ is not defined for $s\in(-\infty,0)$ and, unlike in the complex setting, it is not possible to choose different branches of $s^{\alpha}$ in order to avoid this problem. When we define the fractional powers of $T$ we have to take  this fact into account and choose a suitable path of Integration in Balakrishnan's formula.  We furthermore use the notation  $P_{\alpha}(T)$, for  $\alpha\in (0,1)$,
to stress that  we only take  spectral values $s$  with  $Re(s)\geq 0$ into account,
i.e. only points where $s^{\alpha}$ is actually defined. We will explain this in the sequel with more details in order to justify the formulas (\ref{BALA1}) or equivalently in (\ref{BALA2}).

 The main results are Theorems \ref{Thm1} and \ref{MDUE}  that we summarize in the following.
\begin{theorem} Let $\Omega$ be a bounded domain in $\mathbb{R}^3$ with sufficiently smooth boundary.
Let $T$ be the operator defined in (\ref{TCOM})
 and assume that the coefficients  $a_{\ell}:\overline{\Omega} \to \mathbb{R}$, for $\ell = 1,2,3$ belong to
$ \mathcal{C}^1(\overline{\Omega},\rr)$  and $a_{\ell}(x_\ell)\geq m$ in $\overline{\Omega}$ for some $m>0$.
Moreover, assume that
\[
\inf_{\vx\in\Omega} \left| a_{\ell}(x_{\ell})^2\right| - \frac{\sqrt{C_{\Omega}}}{2}\left\| \partial_{x_{\ell}}a_{\ell}(x_{\ell})^2\right\|_{\infty} > 0, \qquad \ell = 1,2,3,
\]
and
\[
\frac{1}{2} - \frac{1}{2}\|\Phi\|_{\infty}^2C_{\Omega}^2 C_a^2 >0
\]
where $C_\Omega$ is the Poincar\'{e} constant of $\Omega$ and
\[
\Phi(\vx) :=  \sum_{\ell = 1}^{3}e_{\ell} \partial_{x_{\ell}} a_{\ell}(x_{\ell})\quad\text{and}\quad C_a := \sup_{\substack{\vx\in\Omega \\ \ell = 1,2,3}} \frac{1}{| a_{\ell}(x_{\ell})|} =  \frac{1}{\inf_{\substack{\vx\in\Omega \\ \ell = 1,2,3}} | a_{\ell}(x_{\ell})|}.
\]
Then any $s\in\hh\setminus\{0\}$ with $\Re(s) = 0$ belongs to $\rho_S(T)$ and the $S$-resolvents satisfy the estimate
\begin{equation}\label{SREST}
\left\|S_L^{-1}(s,T)\right\| \leq \frac{\Theta}{|s|}\qquad\text{and}\qquad \left\|S_R^{-1}(s,T)\right\| \leq \frac{\Theta}{|s|},\qquad \text{if}\quad\Re(s) =0,
\end{equation}
with a constant $\Theta >0$ that does not depend on $s$.
Moreover, for $\alpha\in(0,1)$, and for any  $v\in\dom(T)$, the integral
\[
P_{\alpha}(T)v := \frac{1}{2\pi}\int_{-\uI\rr} s^{\alpha-1}\,ds_{\uI}\,S_{R}^{-1}(s,T) Tv.
\]
converges absolutely in $L^2(\Omega,\hh)$.
\end{theorem}
Finally we observe that
the operator $P_{\alpha}(T)$ can equivalently be represented using the left $S$-resolvent operator.
\\
\\
For researchers that are interested in quaternionic operator theory we mention that
the most important results in quaternionic operators theory and the associated theory of slice hyperholomorphic function are contained in the books \cite{ACSBOOK,CGKBOOK,ACSBOOK2,MR2752913,GSSb}.
The original papers where the general version of the $S$-functional calculus has been developed are
 \cite{acgs,JGA,CLOSED,Where,CSSJFA,DA}.
From an historical point of view, the research for an appropriate notion of quaternionic spectrum, the $S$-spectrum,
started because in quaternionic quantum mechanics, see  \cite{adler,BvN}, the notion of quaternionic spectrum was unclear.
 Recent investigations in quaternionic quantum mechanics based on the current knowledge in quaternionic operator theory  can be found for example in \cite{JONAQS,IRE1,IRE2}.
The basic tool for quaternionic quantum mechanics is the spectral theorem based on the $S$-spectrum  that
 been developed in \cite{ack,acks2}, and in \cite{MatSpec} for the case of matrices. Beyond the spectral theorem there several directions further directions of research. The case of  spectral operators has for instance been studied in \cite{SpecOP} and the theory of groups and semigroups of quaternionic linear operators has been studied in
\cite{FUCGEN,perturbation,evolution,GR}.
\\
\\
The plan of the paper is as follows: 
Section 1 contains the introduction. In Section 2  we introduce the basic concepts of the spectral theory on the $S$-spectrum and
Section 3 contains a new strategy to fractional diffusion processes. In particular we define in a subsection: "The $S$-spectrum approach to fractional diffusion processes". Finally in Section 4 we give a concrete  application of the $S$-spectrum approach to a Fractional Fourier's law in a Hilbert space.

\section{Basic concepts of the spectral theory on the $S$-spectrum}

In this section we recall the basic definitions of the quaternions and of the spectral theory for quaternionic operators.
The skew-field of quaternions consists of the real vector space
$
\hh:=\left\{p = \xi_0 + \sum_{\ell=1}^3\xi_{\ell}e_{\ell}: \xi_{\ell}\in\rr\right\},
$
 which is endowed with an associative product with unity $1$ such that
$e_{\ell}^2 = - 1$ and $e_\ell e_m = -e_me_\ell $ for $\ell,m\in\{1,2,3\}$ with $\ell\neq m$.
The real part of a quaternion $p = \xi_0 + \sum_{i=1}^3\xi_ie_i$ is defined as $\Re(p) := \xi_0$, its imaginary part as $\underline{p} := \sum_{i=1}^3\xi_i e_i$ and its conjugate as $\overline{p} := \Re(p) - \underline{p}$.
Each element $j$ of the set
\[\SS := \{ p\in\hh: \Re(p) = 0, |p| = 1 \}\]
 satisfies $j^2 = -1$ and is therefore called an imaginary unit. For any $j\in\SS$,
 the subspace $\cc_j := \{p_0 + j p_1: p_0,p_1\in\rr\}$ is an isomorphic copy of the field of complex numbers.   Any quaternion $p$ belongs to such a complex plane $\mathbb{C}_j$: if we set
 \[j_p := \begin{cases}\frac{1}{|\underline{p}|}\underline{p},& \text{if  }\underline{p} \neq 0 \\ \text{any }j\in\SS, \quad&\text{if }\underline{p}  = 0,\end{cases}\]
 then $p = p_0 + j_p p_1$ with $p_0 =\Re(p)$ and $p_1 = |\underline{p}|$. The set
 \[
 [p] := \{p_0 + ip_1: i\in\SS\},
 \]
is a 2-sphere, that reduces to a single point if $p$ is real. We consider the field of real numbers as a subset of $\hh$ and identify it with the set of real quaternions $\rr\cong \{p\in\hh: \underline{p} = 0\}$. We also consider the three-dimensional Euclidean space as a subset of $\hh$ and identify it with the set of purely imaginary quaternions, that is $\rr^3\cong \{q\in\hh: \Re(q) = 0\}$.

The theory of complex linear operators is based on the theory of holomorphic functions.
In a similar way, the theory of quaternionic linear operators is based on the theory of slice hyperholomorphic functions. We recall just some basic facts in order to introduce the tools that will be used in the sequel. The
proofs of the results  can be found in the books  mentioned in the introduction.
\begin{definition}[Slice hyperholomorphic functions]\label{sHolDef}
Let $U\subset\hh$ be open and axially symmetric, that is $[p]\subset U$ for any $p\in U$. A function $f: U \to \hh$ is called left slice hyperholomorphic, if it is of the form
\begin{equation}
\label{lHolDef}f(p) = \alpha(p_0,p_1) + j_p \beta(p_0,p_1) \qquad \forall p\in U,
\end{equation}
where $\alpha$ and $\beta$ are functions that take values in $\hh$, satisfy the compatibility condition
\begin{equation}\label{CCond}
\alpha(p_0,p_1) = \alpha(p_0,-p_1)\quad\text{and}\quad \beta(p_0,p_1) = -\beta(p_0,-p_1)
\end{equation}
and the Cauchy-Riemann-differential equations
\begin{equation}\label{CR}
\frac{\partial}{\partial p_0} \alpha(p_0,p_1) = \frac{\partial}{\partial p_1} \beta(p_0,p_1)\quad\text{and}\quad \frac{\partial}{\partial p_1}\alpha(p_0,p_1) = - \frac{\partial}{\partial p_0} \beta(p_0,p_1).
\end{equation}
A function $f: U \to \hh$ is called right slice hyperholomorphic, if it is of the form
\begin{equation}\label{rHolDef}
f(p) = \alpha(p_0,p_1) +\beta(p_0,p_1)  j_p \quad \forall p\in U,
\end{equation}
with functions $\alpha$ and $\beta$ satisfying \eqref{CCond} and \eqref{CR}. Finally, a left slice hyperholomorphic function $f = \alpha + j \beta$ is called intrinsic, if $\alpha$ and $\beta$ are real-valued.
\end{definition}
Slice hyperholomorphic functions satisfy Cauchy formulas with left resp. right slice hyperholomorphic kernels.
\begin{definition}
For $s,p\in\hh$ with $s\notin [p]$, we define 
\[
\Q_{s}(q) := (q^2 - 2\Re(s) q + |s|^2).
\]
The left and right slice hyperholomorphic Cauchys kernel are defined as
\[
S_L^{-1}(s,p) := - \Q_{s}(p)^{-1}(p-\overline{s})\quad\text{and}\quad S_{R}^{-1}(s,p) := -(p-\overline{s})\Q_{s}(p)^{-1}.
\]
\end{definition}
\begin{theorem}
Let $U$ be a bounded slice Cauchy domain, that is $U$ is an bounded, open and axially symmetric set such that $\partial (U\cap \cc_j)$ consists of a finite number of Jordan curves for some (and hence any) $j\in\SS$. If $f$ is left slice hyperholomorphic on $\clos{U}$, then  
\[
f(p) = \frac{1}{2\pi}\int_{\partial (U\cap\cc_j)} S_L^{-1}(s,p)\,ds_j\,f(s)\qquad \text{for any }p\in U,
\]
where $j\in\SS$ can be chosen arbitrarily and $ds_j := ds(-j)$. 
 If $f$ is right slice hyperholomorphic on $\clos{U}$, then  
\[
f(p) = \frac{1}{2\pi}\int_{\partial (U\cap\cc_j)} f(s) \,ds_j\,S_L^{-1}(s,p)\qquad \text{for any }p\in U,
\]
where $j\in\SS$ can be chosen arbitrarily and $ds_j := ds(-j)$. 
\end{theorem}

The slice-hyperholomorphic Cauchy formulas are the starting point for the definition of the the $S$-functional calculus, which is the natural extension of the Riesz-Dunford-functional calculus for complex linear operators  to quaternionic linear operators. It is based on the theory of slice hyperholomorphic functions and follows the idea of the classical case: we formally replace the scalar variable $p$ in the slice hyperholomorphic Cauchy formula by an operator $T$. The proofs of the results stated in this subsection can be found in \cite{acgs, MR2752913, DA}. We recall the main definitions related to the $S$-functional calculus.
The $H^\infty$-functional calculus, which is able to define unbounded functions of operators, is too complicated to be summarized here and we refer the interested reader to to the original papers, see \cite {Hinfty,64FRAC}.
Here we just use  the quaternionic  Balakrishnan's formula that is a consequence of the $H^\infty$ functional calculus developed in \cite{64FRAC}.

Let $V$ denote a two-sided quaternionic Banach space. We denote the set of all bounded quaternionic right-linear operators on $V$ by $\boundOP(V)$ and the set of all closed and quaternionic right-linear operators on $V$ by $\closOP(V)$.
One of the major differences between the theory of complex linear operators and the theory of quaternionic linear operators, and particular of vector operators, is that the spectrum for quaternionic operators is defined by a second order operator. Precisely,
for $T\in\closOP(V)$, we define the operator
\[
\Q_{s}(T) := T^2 - 2\Re(s)T + |s|^2\id: \ \mathcal{D}(T^2)\subset V\to V, \qquad \text{for $s\in\hh$}.
\]
\begin{definition}
Let $T\in\closOP(V)$. We define the $S$-resolvent set of  $T$ as
\[\rho_S(T):= \{ s\in\hh: \Q_{s}(T)^{-1}\in\boundOP(V)\}\]
and the $S$-spectrum of $T$ as
\[\sigma_S(T):=\hh\setminus\rho_S(T).\]
If $s\in\rho_S(T)$, then $\Q_{s}(T)^{-1}$ is called the pseudo-resolvent of $T$ at $s$. Furthermore, we define the extended $S$-spectrum $\sigma_{SX}(T)$ as
\begin{equation*}
\sigma_{SX}(T) := \begin{cases} \sigma_{S}(T) & \text{if  $T$ is bounded,}\\
\sigma_{S}(T)\cup\{\infty\} & \text{if  $T$ is unbounded.}
\end{cases}
\end{equation*}
\end{definition}

The $S$-spectrum has properties that are similar to those of the spectrum of a complex linear operator.
\begin{theorem}\label{SpecProp}
Let $T\in\closOP(V)$.
\begin{enumerate}[label = (\roman*)]
\item The $S$-spectrum $\sigma_{S}(T)$ of $T$ is axially symmetric.
\item The $S$-spectrum $\sigma_{S}(T)$ is a closed subset of $\hh$ and the extended $S$-spectrum $\sigma_{SX}(T)$ is a closed and hence compact subset of $\hh_{\infty}:= \hh \cup\{\infty\}$.
\item If $T$ is bounded, then $\sigma_{S}(T)$ is nonempty and bounded by the norm of $T$, i.e. $\sigma_{S}(T)\subset \overline{B_{\|T\|}(0)}$, where $B_{\|T\|}(0)$ is the open ball of radius $\|T\|$ centered at $0$.
\end{enumerate}
\end{theorem}
A second important fact is that the pseudo-resolvent  $\Q_{s}(T)^{-1}$ of $T$ is not the resolvent of the $S$-functional calculus. In particular,  the operator-valued function $s\mapsto\Q_{s}(T)^{-1}$ is not  slice hyperholomorphic.
Instead, we have two different $S$-resolvent operators.
\begin{definition}\label{SResDef}
Let $T\in\closOP(V)$. For $s\in\rho_S(T)$, the left $S$-resolvent operator is defined as
\begin{equation}\label{SresolvoperatorL}
S_L^{-1}(s,T):= \Q_s(T)^{-1}\overline{s} -T\Q_s(T)^{-1}
\end{equation}
and the right $S$-resolvent operator is defined as
\begin{equation}\label{SresolvoperatorR}
S_R^{-1}(s,T):=-(T-\id \overline{s})\Q_s(T)^{-1}.
\end{equation}
\end{definition}
The $S$-resolvent operators reduce to the classical resolvent if $T$ and $s$ commute, that is
\[S_L^{-1}(s,T) = S_R^{-1}(s,T) = (s\id - T)^{-1}.\]
Furthermore, the right $S$-resolvent is obviously obtained by formally replacing $p$ by $T$ in the right slice hyperholomorphic Cauchy kernel. Formally replacing $p$ by $T$ in the left Cauchy kernel yields the operator $-\Q_{s}(T)^{-1}(T-\overline{s}\id)$, which is only defined on $\dom(T)$. The left $S$-resolvent defined in \eqref{SresolvoperatorL} is the closed extension of this operator and defined on all of $V$. In particular, if $T$ is bounded, then $\dom(T) = V$ and so $S_L^{-1}(s,T) = -\Q_{s}(T)^{-1}(T-\overline{s}\id)$.

The following lemma is crucial and can be shown by straightforward computations for bounded operators. In the case of unbounded operators, several additional technical difficulties have to be overcome. The respective proof can be found in \cite{FJTAMS}.
\begin{theorem}
Let $T\in\closOP(V)$. The map $s\mapsto S_L^{-1}(s,T)$ is a right slice-hyperholmorphic function on $\rho_S(T)$ with values in the two-sided  quaternionic Banach space $\boundOP(V)$. The map $s\mapsto S_R^{-1}(s,T)$ is a left slice-hyperholmorphic function on $\rho_S(T)$ with values in the two-sided quaternionic Banach space $\boundOP(V)$.
\end{theorem}
The $S$-resolvent equation is the analogue of the classical resolvent equation in the quaternionic setting. It has first been proved for the case that $T$ is a bounded operator in \cite{acgs}. Later on it was generalized to the case that $T$ is unbounded in \cite{FJTAMS}.
We recall the $S$-resolvent equation because
it is one of the main differences between the complex and the quaternionic case. Observe that it involves both the $S$-resolvent operators.

\begin{theorem}[$S$-resolvent equation]Let $T\in\closOP(V)$. For  $s,p \in  \rho_S(T)$ with $s\notin[p]$, it is
\begin{equation}\label{SresEQ1}
\begin{split}
S_R^{-1}(s,T)S_L^{-1}(p,T)=&\big[[S_R^{-1}(s,T)-S_L^{-1}(p,T)]p
\\
&
-\overline{s}[S_R^{-1}(s,T)-S_L^{-1}(p,T)]\big](p^2-2s_0p+|s|^2)^{-1}.
\end{split}
\end{equation}
\end{theorem}

Any vector $v$ in the two-sided quaternionic vector space $V$ can be written in a unique way as $v = v_0 + \sum_{\ell = 1}^3 v_{\ell} e_{\ell}$ with components $v_{\ell}$ in the real Banach space $V_{\rr} := \{v\in V: a v = va\ \forall a\in\hh\}$. We write $V = V_{\rr}\otimes \hh$ and say that $V$ is the quaternionification of $V_{\rr}$. Any bounded operator $T$ on $V$ can then be written as $T = T_0 + \sum_{\kappa=1}^{3} T_{\kappa}e_{\ell}$ with components $T_{\ell}\in\boundOP(V_{\rr})$ i.e. they are bounded $\rr$-linear operators on $V_{\rr}$. We then have
\[
T v = \sum_{\kappa,\ell=0}^3 e_{\kappa}e_{\ell} T_{\kappa}v_{\ell},
\]
where we set $e_0 := 1$ for neatness. If $T$ is unbounded, then we can only write $T$ in terms of $\rr$-linear components if $\dom(T)$ is of the form $\dom(T) = \{ \sum_{\ell} v_{\ell}e_{\ell}: v_{\ell}\in W\}$ for some $\rr$-linear subspace $W$. Then $T = T_0 + \sum_{\kappa = 1}^3 T_{\kappa}e_{\kappa}$ with $\rr$-linear components $T_{\kappa}: W \to V_{\rr}, \kappa = 0,\ldots, 3$. Conversely, if $T_{0},\ldots,T_{3}$ are $\rr$-linear operators on $V$, then $T = \sum_{\kappa=0}^3T_{\kappa}e_{\kappa}$ is a quaternionic linear operator with domain $\dom(T) = \bigcap_{\kappa = 0}^3 \dom(T_{\kappa}) \otimes\hh$.

For operators with commuting components, the $S$-resolvent operators admits a second representation and the $S$-spectrum can be computed more easily.  Precisely, we consider the class $\closOP(V)$ which consists of those operators  $ T= T_0 + \sum_{\ell=1}^{3}T_{\ell}e_{\ell}$ such that the components
$T_{\ell}:\mathcal{D}(T_\ell) \subset V_{\rr}\to V_{\rr}$, $\ell = 0,\ldots,3$ are closable operators on $V_{\rr}$
that  commute mutually, that is $T_{\ell} T_{\kappa}=T_{\kappa} T_{\ell}$ on $\mathcal{D}(T_{\ell}T_{\kappa})\cap \mathcal{D} (T_{\kappa}T_{\ell})$ for $\ell,\kappa =0,\ldots, 3$.
 The domain of $T$ is $\mathcal{D}(T)= \cap_{\ell=0}^3 \mathcal{D}(T_\ell)$ and if we set
\[
\overline{T} := T_0 - \sum_{\ell = 1}^3 T_{\ell}e_{\ell},
\]
then the domain of $\overline{T}$  is $\mathcal{D}(\overline{T})= \mathcal{D}({T})$.  For $s\in\hh$, we define the commutative pseudo-resolvent operator as
\[
\Q_{c,s}(T) := s^2 \id - 2s T_0  + T\overline{T}: \ \ \ \mathcal{D}(T^2)\subset V\to V.
\]
Observe that the operators  $T_0$ and $T\overline{T}  = \sum_{\kappa=0}^3 T_{\kappa}^2$ do not contain the imaginary units of the quaternions. The operator $\Q_{c,s}(T)$ in invertible if and only if $\Q_{s}(T)$ is invertible and so
\begin{equation}\label{SSpecCommut}
\rho_S(T) = \left\{ s\in \hh : \Q_{c,s}(T)^{-1}\in\boundOP(X) \right\}.
\end{equation}
Moreover, for $s\in\rho_{S}(T)$, the $S$-resolvent operators can be written as
\begin{align}\label{ZwoaDreiVia}
S_L^{-1}(s,T) = &(s\id- \overline{T})\Q_{c,s}(T)^{-1}\\
\intertext{and}
\label{ZwoaDreiVia1}
 S_R^{-1}(s,T) = & \Q_{c,s}(T)^{-1}s - \sum_{\kappa = 0}^3 T_{\kappa}\Q_{c,s}(T)^{-1}\overline{e_{\kappa}}.
\end{align}

Using the above definitions and the quaternionic $H^\infty$-functional calculus, which is an extension of the $S$-functional calculus for sectorial operators, we can define fractional powers of closed sectorial operators, i.e. operators $T$ such that $(-\infty,0)\subset \rho_{S}(T)$ and $\|S_{L}^{-1}(s,T)\| = \|S_{R}^{-1}(s,T) \leq C/|s|$ for some $C>0$ and any $s\in (-\infty,0)$. For $\alpha \in (0,+\infty)$, such fractional power can be represented via Balakrishnan's formula
\begin{equation}\label{BII}
T^{\alpha}v = -\frac{\sin(\alpha\pi)}{\pi}\int_{0}^{+\infty}t^{\alpha-1}S_{R}^{-1}(s,T)Tv \,dt\qquad\forall v\in \dom(T). 
\end{equation}

Unfortunately, the gradient  $\nabla$ on $L^2(\rr^3,\hh)$ does not turn out to be a quaternionic sectorial operator: its $S$-spectrum contains only real points and whenever a positive point $s>0$ belongs to $\sigma_S(T)$, then also $-s$ belongs to $\sigma_S(T)$ so that for sure $(-\infty,0)\not\subset \rho_S(T)$. Moreover, the princial branch of $s^{\alpha}$ is not defined on $(-\infty,0]$, which contains spectral values of $\nabla$. Hence, defining $\nabla^{\alpha}$ is not possible by simply applying the existing theory and using formula \eqref{BII}. For complex linear operators, one could try to choose a different branch of the fractional power $s^{\alpha}$ that is defined on $(-\infty,0]$ and then use this branch to define $\nabla^{\alpha}$ using a suitable version of the $S$-functional calculus or the quaternionic $H^{\infty}$-functional calculus.
Choosing different branches of fractional powers is however not possible in the quaternionic setting. Instead we had to choose a modified approach. 

The slice hyperholomorphic logarithm on $\hh$ is defined as
\[
\log s:=\ln |s|+j_s \arg(s)\qquad\text{for } s\in \mathbb{H}\setminus (-\infty, 0],
\]
where $\arg(s) =  \arccos(s_0/|s|)$ is the unique angle $\varphi\in[0,\pi]$ such that $s = |s| e^{j_s\varphi}$.
Observe that for $s=s_0\in[0,\infty)$ we have  $\arccos(s_0/|s|) = 0$ and so $\log s = \ln s$. Therefore, $\log s$ is well defined on the positive real axis and does not depend on the choice of the imaginary unit $j_s$. One has
\[ e^{\log s} = s  \quad\text{for } s\in\hh\]
and
\[ \log e^s = s \quad\text{for }s\in\hh\ \ \text{ with }|\underline{s}|<\pi.\]
The quaternionic logarithm is both left and right slice hyperholomorphic (and actually even intrinsic) on $\hh\setminus (-\infty,0]$ and for any $j\in\SS$ its restriction to the complex plane $\cc_j$ coincides with the principal  branch of the complex logarithm on $\cc_j$.
We define the fractional powers of exponent $\alpha\in \mathbb{R}$ of a quaternion $s$ as
\[
s^{\alpha}:= e^{\alpha\log s}=e^{\alpha (\ln |s|+j_s \arccos(s_0/|s|))}, \ \ \  s\in \mathbb{H}\setminus (-\infty, 0].
\]
This function is obviously also left and right slice hyperholomorphic on $\hh\setminus(-\infty, 0]$. Note however that if we define fractional powers $s^{\alpha}$ with $\alpha\in\hh\setminus\rr$ by the above formula, we do not obtain a slice hyperholomorphic function because the composition of two slice hyperholomorphic functions is in general only slice hyperholomorphic if the inner function is intrinsic.
We choose now  $j\in\SS$ and set for $\alpha\in(0,1)$ 
\begin{equation}\label{Palpha1}
P_\alpha(T) v := \frac{1}{2\pi} \int_{-j\rr}   S_L^{-1}(s,T)\,ds_j\, s^{\alpha-1} T v,\ \ \ v\in \mathcal{D}(T).
\end{equation}
This corresponds to Balakrishnan's formula for the $T^{\alpha}$, where just a part of the $S$-spectrum is taken into account, namely the part of the spectrum contained in the positive real axis, where the function $s^{\alpha}$ is actually defined.
An analogue Balakrishnan's formula for $P_\alpha(T)$ holds using the right $S$-resolvent operator $S_R^{-1}(s,T)$
\begin{equation}\label{Palpha2}
P_\alpha(T) v := \frac{1}{2\pi} \int_{-j\rr} s^{\alpha-1}  \,ds_j\, S_R^{-1}(s,T) T v, \ \ \ v\in \mathcal{D}(T),
\end{equation}
 and it defines the same operator $P_\alpha(T)$.

\section{A new strategy to fractional diffusion processes}

Using the formulas \eqref{Palpha1} and \eqref{Palpha2}, we can now reduce the problem
of defining the fractional powers of a vector operator $T$ to a problem in partial differential equations
in which we show that $S$-resolvent operators $S_L^{-1}(s,T)$ and $S_R^{-1}(s,T)$ are well defined and satisfy suitable growth conditions in $|s|$
when  the coefficients that define the vector operator $T$ satisfy suitable estimates.
The procedure is not immediate  and will be explained in some steps.
\\
\\
First we consider the case when $T$ is minus the gradient operator because in this case we expect that our new method,
based on the spectral theory on the $S$-spectrum, reproduces the classical fractional heat equation that contains the fractional powers of the negative Laplace operator. So
we consider the heat flux $\mathbf{q}(\nabla)  = -\nabla v$, where $v$ is the temperature, we identify
$
\mathbb{R}^3 \ \cong \{ s\in\mathbb{H}: \Re (s) = 0\}
$
and we consider the gradient operator $\nabla$ as the quaternionic  operator
\begin{equation}\label{GRD}
\mathbf{q}(\nabla)v  = -(e_1\partial_{x_1} + e_2 \partial_{x_2} +  e_3\partial_{x_3})v.
\end{equation}
Instead of replacing the negative Laplacian $(-\Delta)^{\alpha}$ in the heat equation, we replace the fractional gradient in  the equation
$$
\partial_t v(t,x) +\operatorname{div} \mathbf{q}(\nabla)v(t,x) = 0.
$$
The following two observations are of crucial importance in order to define the new procedure to fractional diffusion processes.
\begin{itemize}
\item[(I)]
Since  $s^{\alpha}$, for $\alpha\in (0,1)$,  is not defined on $(-\infty,0)$
and  on $L^2(\mathbb{R}^3,\mathbb{H})$ it is $\sigma_{S}(\nabla) = \mathbb{R}$
we consider the projections of the fractional powers of $\nabla^{\alpha}$, indicated by  $P_{\alpha}(\nabla)$, to the subspace associated with the subset $[0,+\infty)$ of the $S$-spectrum of $\nabla$, on which the function $s^{\alpha}$ is well defined and it is slice hyperholomorphic.
\item[(II)]
The above procedure gives a quaternionic operator
$$
P_{\alpha}(\nabla)=T_0+e_1T_1+e_2T_2+e_3T_3,
$$
 where $T_\ell$, $\ell=0,1,2,3$ are real operators obtained by the functional calculus.
Finally, we take the vector part
$$
{\rm Vect}(P_{\alpha}(\nabla))=e_1T_1+e_2T_2+e_3T_3
$$ of the quaternionic operator $P_{\alpha}(\nabla)$
so that  we can apply the divergence operator.
\end{itemize}

More explicitly we
define  $\nabla^{\alpha}$ only on the subspace associated to $[0,\infty)$  that is
\[
P_{\alpha}(\nabla) v = \frac{1}{2\pi} \int_{-j\mathbb{R}} S_L^{-1}(s,\nabla)\,ds_j\, s^{\alpha-1} \nabla v,
\]
for $v:\mathbb{R}^3\to \mathbb{R}$  in $\dom(\nabla)$, where the path integral is computed taking into account just the part of the $S$-spectrum with
${\rm Re}(\sigma_S(\nabla))\geq 0$ since  we have proven in \cite{64FRAC} that the $S$-spectrum of the gradient is the entire real line.
With this definition and the surprising expression for the left $S$-resolvent operator
\[
S_L^{-1}(-jt ,\nabla) = (-jt + \nabla)\underbrace{(-t^2 + \Delta)^{-1}}_{= R_{-t^2}(-\Delta)},
\]
 the fractional powers $P_{\alpha}(\nabla)$ become
$$
P_{\alpha}(\nabla) v
= \underbrace{ \frac{1}{2} (-\Delta)^{\frac{\alpha}{2}-1}\nabla^2 v}_{\mathrm{Scal} P_{\alpha}(\nabla)v} + \underbrace{\frac{1}{2} (-\Delta)^{\frac{\alpha-1}{2}}\nabla v}_{=\mathrm{Vec}P_{\alpha}(\nabla) v}.
$$
We define the scalar part  of the operator $P_{\alpha}(\nabla) $ applied to $v$ as
$$
\mathrm{Scal} P_{\alpha}(\nabla)v:=\frac{1}{2}  (-\Delta)^{\frac{\alpha}{2}-1}\nabla^2 v,
$$
and the vector part as
$$
\mathrm{Vec}P_{\alpha}(\nabla) v:=\frac{1}{2} (-\Delta)^{\frac{\alpha-1}{2}}\nabla v.
$$
Now we observe hat
\[
 {\rm div} \mathrm{Vec}P_{\alpha}(\nabla) v = -\frac{1}{2} (-\Delta)^{\frac{\alpha}{2}+1} v
\]
This shows that in the case of the gradient we get the same result, that is the fractional Laplacian.
 The fractional heat equation for $\alpha\in(1/2,1)$
\[
\partial_t v(t,x) + (-\Delta)^{\alpha} v(t,x) = 0
\]
 can hence be written as
\[
\partial_t v(t,x) - 2 {\rm div}\left(\mathrm{Vec} P_{\beta}(\nabla)v\right)  = 0, \qquad \beta = 2\alpha - 1
\]
so we have shown that this approach coincided with the classical one.
\\
\\
Before we state our new approach to fractional diffusion processes we will make some considerations
regarding the case when the components of the operator $T$ commute among themselves and when they do not commute.
When we consider non homogeneous materials in $\Omega \subseteq \mathbb{R}^3$, where $\Omega$ can be bounded or unbounded, then
  Fourier's law becomes
\begin{equation}\label{TNONCOM}
T:=\mathbf{q}({x},\partial_{x})= e_1 \,a(x)\partial_{x_1} +  e_2\, b(x)\partial_{x_2} +  e_3\, c(x)\partial_{x_3},\ \ \ x=(x_1,x_2,x_3)
\end{equation}
where $e_\ell$, $\ell=1,2,3$ are orthogonal unit vectors in $\mathbb{R}^3$ and
the coefficients $a$, $b$, $c:\Omega \to \mathbb{R}$ depend on the space variables $x=(x_1,x_2,x_3)$, and possibly on time.
In this case to define the fractional powers of $T$ we need to show that the pseudo $S$-resolvent operator
\begin{equation}\label{TCOM}
\Q_{s}(T):=T^2 - 2\Re(s)T + |s|^2\id
\end{equation}
 is invertible. In this case the computations can be quite complicated and will considered in a future publication. In this paper we consider the commutative case that gives a better understanding of the our new procedure as we will see in the following.
 \\
 \\
 When $T$ has commuting components, that is $T$ is a vector operator of the form
$$
T=e_1\,a_1(x_1)\partial_{x_1} +  e_2\,a_2(x_2)\partial_{x_2} +  e_3\,a_3(x_3)\partial_{x_3}
$$
where $a_1$, $a_2$, $a_3:\Omega \to \mathbb{R}$ are suitable real valued function that depend on the space variables $x_1$, $x_2$, $x_3$, respectively, then the $S$-spectrum can also be determined by the commutative pseudo-resolvent operator
$$
Q_{c,s}(T) := s^2 \id - 2s T_0  + T\overline{T}=a_1^2(x_1)\partial^2_{x_1} +  a_2^2(x_2)\partial^2_{x_2} +  a_3^2(x_3)\partial^2_{x_3}+s^2\mathcal{I}
$$
because $\Q_{c,s}(T)$ is invertible if and only if $\Q_{s}(T)$ is invertible. The operator $\Q_{c,s}(T)$ is a scalar operator
if $s^2$ is a real number, since $T$ is a vector operator, we have $T_0=0$ and $T\overline{T}$ does not contain the imaginary units of the quaternions.
Using the non commutative expression of the pseudo-resolvent operator
$\Q_{s}(T)$
 we obtain
\[
\begin{split}
\Q_{s}(T)  &=
- (a_1(x_1)\partial_{x_1})^2 - (a_2(x_2)\partial_{x_2})^2 - (a_3(x_3)\partial_{x_3})^2
\\
&
-2s_0(e_1\,a_1(x_1)\partial_{x_1} +  e_2\,a_2(x_2)\partial_{x_2} +  e_3\,a_3(x_3)\partial_{x_3})
+ |s|^2\mathcal{I}.
\end{split}
\]
We observe that according to what we need to show in the commutative case we have two possibilities.
In the next subsection we explicitly write the procedure to defined the $S$-spectrum approach to fractional diffusion processes.

\subsection{ The $S$-spectrum approach to fractional diffusion processes }
\
\medskip

Suppose that $\Omega \subseteq \mathbb{R}^3$ is a suitable bounded or unbounded domain and let $V$ be a two sided Banach space.
We consider  the initial-boundary value problem for non-homogeneous materials. We use the notation $T=\mathbf{q}(x,\partial_x)$ and
we restrict ourselves to the case of homogeneous boundary conditions:
\[
\begin{split}
&T(x) = a_1(x_1)\partial_{x_1} e_1 + a_2(x_2) \partial_{x_2}e_2 + a_3(x_3)\partial_{x_3}e_3,\ \ \ x=(x_1,x_2,x_3)\in \Omega
\\
&
\partial_t v(x,t)  +{\rm div}\, T(x) v(x,t) = 0,\ \ \ \ (x,t)\in \Omega\times (0,T]
\\
&
v(x,0)=f(x),\ \ \ x\in \Omega
\\
&
v(x,t)=0,\ \ x\in \partial \Omega \ \ \ \ t\in [0,T].
\end{split}
\]
Our general procedure consists in the following steps:
\begin{itemize}
\item [(S1)] We study the invertibility of the operator
$$
Q_{c,s}(T) := s^2 \id - 2s T_0  + T\overline{T}=a_1^2(x_1)\partial^2_{x_1} +  a_2^2(x_2)\partial^2_{x_2} +  a_3^2(x_3)\partial^2_{x_3}+s^2\id
$$
where $\overline{T}=-T$, to get the $S$-resolvent operator.
Precisely, let $F:\Omega\to \mathbb{H}$  be a given function with a suitable regularity and
denote by  $X:\Omega\to \mathbb{H}$ the unknown function of the boundary value problem:
\[
\begin{split}
&\Big(a_1^2(x_1)\partial^2_{x_1} +  a_2^2(x_2)\partial^2_{x_2} +  a_3^2(x_3)\partial^2_{x_3}+s^2 \id\Big)X(x)=F(x),\ \ \ x\in \Omega
\\
&
X(x)=0,\ \ x\in \partial \Omega.
\end{split}
\]
We study under which conditions on the coefficients $a_1$, $a_2$, $a_3:\mathbb{R}^3\to \mathbb{R}$ the above equation has a unique solution.
We can similarly use the non commutative version of the pseudo-resolvent operator $\Q_{s}(T)$.
In the case we deal with an operator $T$ non commuting components then we have to consider $\Q_{s}(T)$, only.
\item[(S2)]
From (S1) we get the unique pesudo-resolvent operator $\Q_{c,s}(T)^{-1}$ and so
can define the $S$-resolvent operator
$$
S_L^{-1}(s,T) = (s\id- \overline{T})\Q_{c,s}(T)^{-1}.
$$
Then we prove
that any $s\in\hh\setminus\{0\}$ with $\Re(s) = 0$ belongs to $\rho_S(T)$ and the $S$-resolvent operators satisfy the estimates
\begin{equation}\label{SREST}
\left\|S_L^{-1}(s,T)\right\|_{\mathcal{B}(V)} \leq \frac{\Theta}{|s|}\quad\text{and}\quad 
\left\|S_R^{-1}(s,T)\right\|_{\mathcal{B}(V)} \leq \frac{\Theta}{|s|},\qquad \text{if }\qquad \Re(s) =0,
\end{equation}
with a constant $\Theta >0$ that does not depend on $s$.
\item[(S3)]
 Using the Balakrishnan's formula we define we define
 $P_{\alpha}(T)$ as:
\[
P_{\alpha}(T) v = \frac{1}{2\pi} \int_{-j\mathbb{R}} s^{\alpha-1}\,ds_j\, S_R^{-1}(s,T) T v, \ \ \ {\rm for}\ \  \alpha\in (0,1),
\]
 and $v\in \mathcal{D}(T)$. Analogously,  one can use the definition of $P_{\alpha}(T)$ related to  the left $S$-resolvent operator.
\\
\item[(S4)]
After we define the fractional powers $P_{\alpha}(T)$ of the vector operator $T$, we consider its vector part
$\VEC(P_\alpha(T))$ and we obtain the fractional evolution equation:
$$
\partial_t v(t,x) - \operatorname{div} ( \VEC(P_\alpha(T)v)(t,x) = 0.
$$
\end{itemize}
This approach has several advantages:
\begin{itemize}
\item[(I)] It  modifies the Fourier law but keeps the law of conservation of energy.
\item[(II)] It is applicable to a large class of operators that includes the gradient but also operators with variable
coefficients.
\item[(III)]
The fractional powers of the operator $T$  is more realistic for non homogeneous materials.
 \item[(IV)] The fact that we keep the evolution equation in divergence form allows an immediate definition of the weak solution of the
fractional evolution problem.
\end{itemize}
In the next section we show explicitly how  the $S$-spectrum approach to fractional diffusion processes
 works in the Hilbert spaces setting.

\section{Fractional Fourier's law in a Hilbert space}

In this section we  show that under suitable conditions on the coefficients $a_\ell(x_\ell)$ for $\ell=1,2,3$
we can define the fractional law associated with the quaternionic operator
\[
T = e_1a_1(x_1)\partial_{x_1} + e_2a_2(x_2)\partial_{x_2} + e_3a_3(x_3)\partial_{x_3}.
\]
 The conditions $a_\ell(x_\ell)$ for $\ell=1,2,3$ are determined by the Lax-Milgram lemma since we work
in a quaternionic Hilbert space.
We define
\[
L^2:= L^2(\Omega,\hh) := \left\{u: \Omega\to\hh: \int_{\Omega}|u(\vx)|^2\,d\vx < + \infty\right\}
\]
with  the scalar product
\[
\langle u,v\rangle_{L^2} := \langle u,v\rangle_{L^2(\Omega,\hh)} := \int_{\Omega} \overline{u(\vx)}v(\vx)\,d\vx,
\]
 where $u(x)=u_0(x)+u_1(x)e_1+u_2(x)e_2+u_3(x)e_3$  and $v(x)=v_0(x)+v_1(x)e_1+v_2(x)e_2+v_3(x)e_3$
for $\vx = (x_1,x_2,x_3)\in \Omega$ and $\Omega\subset\rr^3$ is an open set with smooth boundary.
 We furthermore introduce the quaternionic Sobolev space
 \[
 H^1:= H^1(\Omega,\hh): = \Big\{u\in L^2(\Omega,\hh) : \exists \ g_{\ell,j}(x) \in L^2(\Omega,\mathbb{R}), \ell = 1,2,3, \ j=0,1,2,3
 \]
 \[
{\rm  such\ \  that\ }
 \int_\Omega \partial_{x_\ell}u_j(x)\varphi(x)dx=- \int_\Omega g_{\ell,j}(x)\partial_{x_\ell}\varphi(x)dx, \ \ \forall \varphi\in \mathcal{C}^\infty_c(\Omega,\mathbb{R})
 \Big\}
 \]
 where $\mathcal{C}^\infty_c(\Omega,\mathbb{R})$ is the set of real-valued infinitely differentiable functions with compact support on $\Omega$. With the quaternionic scalar product
 \[
 \langle u,v\rangle_{H^1} := \langle u, v \rangle_{H^1(\Omega,\hh)} := \langle u,v\rangle_{L^2} + \sum_{\ell = 1}^3 \left\langle \partial_{x_{\ell}}u,\partial_{x_{\ell}}v\right\rangle_{L^2} .
 \]
 $H^1$ is a quaternionic Hilbert space.
 We define
$H_0^1:= H_0^1(\Omega,\hh)$ to be the closure of $\mathcal{C}^\infty_c(\Omega,\mathbb{H})$ in $H^1(\Omega,\hh)$.
 As in the scalar case, see \cite[Theorem 9.17]{BREZIS} if we suppose that $\Omega$ is of class $\mathcal{C}^1$ and we assume that
 $$
 u\in H^1(\Omega,\hh)\cap \mathcal{C}(\overline{\Omega},\hh)
 $$
 then the condition $u=0$ on $\partial \Omega$ is equivalent to $u\in H_0^1(\Omega,\hh)$.
In the general case , for arbitrary functions $u\in H^1(\Omega,\hh)$, we have to consider the trace operator  $u\mapsto u|_{\partial \Omega}$,
and the space $H_0^1$ turns out to be the kernel of the trace operator, i.e.
  \[
 H_0^1:= H_0^1(\Omega,\hh) := \left\{ u\in H^1(\Omega,\hh): u|_{\partial \Omega} = 0 \right\},
 \]
 where $u|_{\partial \Omega} := \mathrm{tr}\, u$ has to be understood in the sense of the  trace operator, see  \cite[p. 315]{BREZIS}.
  Finally,  $H_0^1(\Omega,\hh)$
  is a subspace of $H^1(\Omega,\hh)$ that is a quaternionic Hilbert space itself. We define
 \[
 \|u\|_D^2 := \sum_{\ell = 1}^3 \left\| \partial_{x_{\ell}}u\right\|_{L^2} ^2.
 \]
Due to the regularity of $\partial\Omega$, the Poincar\'{e}-inequality
\[
\|u\|_{L^2} \leq C_{\Omega} \| u\|_{D}
\]
holds and so $\|u\|_{D}$ defines a norm on $H_0^1(\Omega,\hh)$ that is equivalent to $\|u\|_{H^1}$. However, we want to point out that $\|u\|_{D} \neq \|\nabla u\|_{L^2}$ if $\nabla$ denotes the quaternionic nabla operator $\nabla = \partial_{x_1}e_1 + \partial_{x_2}e_2 + \partial_{x_3}e_3$ although this notation is used for Sobolev spaces of real-valued functions.
From the proof of the Poincar\'{e} inequality, it is obvious that one can even choose $C_{\Omega}$ such that also
\begin{equation}\label{SpecPoin}
\|u\|_{L^2} \leq  C_{\Omega} \| \partial_{x_{\ell}}u\|_{L^2},\qquad\ell = 1,2,3.
\end{equation}
The first result we need to show is that all purely imaginary quaternions  are in the $S$-resolvent set of $T$.
 Moreover, we show that the $S$-resolvent operators decay sufficiently fast at infinity so that operator defined in $(S3)$, that is:
\[
P_{\alpha}(T) v = \frac{1}{2\pi} \int_{-j\mathbb{R}} s^{\alpha-1}\,ds_j\, S_R^{-1}(s,T) T v, \ \ \ v\in \mathcal{D}(T)
\]
 for $\alpha\in (0,1)$, turns out to be well defined.
\begin{theorem}\label{Thm1}
Let $\Omega$ be a bounded domain in $\mathbb{R}^3$ with smooth boundary.
Assume that the coefficients  $a_{\ell}:\overline{\Omega} \to \mathbb{R}$, for $\ell = 1,2,3$ of the operator
$$
T = e_1a_1(x_1)\partial_{x_1} + e_2a_2(x_2)\partial_{x_2} + e_3a_3(x_3)\partial_{x_3}
$$
 belong to $\mathcal{C}^1(\overline{\Omega},\rr)$  and suppose that
\[
\inf_{\vx\in\Omega} \left| a_{\ell}(x_{\ell})^2\right| - \frac{\sqrt{C_{\Omega}}}{2}\left\| \partial_{x_{\ell}}a_{\ell}(x_{\ell})^2\right\|_{\infty} > 0, \qquad \ell = 1,2,3,
\]
and
\[
\frac{1}{2} - \frac{1}{2}\|\Phi\|_{\infty}^2C_{\Omega}^2 C_a^2 >0
\]
where $C_\Omega$ is the Poincar\'{e} constant of $\Omega$ satisfying \eqref{SpecPoin} and
\[
\Phi(\vx) :=  \sum_{\ell = 1}^{3}e_{\ell} \partial_{x_{\ell}} a_{\ell}(x_{\ell})\quad\text{and}\quad C_a := \sup_{\substack{\vx\in\Omega \\ \ell = 1,2,3}} \frac{1}{| a_{\ell}(x_{\ell})|} =  \frac{1}{\inf_{\substack{\vx\in\Omega \\ \ell = 1,2,3}} | a_{\ell}(x_{\ell})|}.
\]
Then any $s\in\hh\setminus\{0\}$ with $\Re(s) = 0$ belongs to $\rho_S(T)$ and the  $S$-resolvent operators satisfy the estimates
\begin{equation}\label{SREST}
\left\|S_L^{-1}(s,T)\right\|_{\mathcal{B}(L^2)} \leq \frac{\Theta}{|s|}
\quad\text{and}\quad \left\|S_R^{-1}(s,T)\right\|_{\mathcal{B}(L^2)} \leq \frac{\Theta}{|s|},\qquad \text{if }\Re(s) =0,
\end{equation}
with a constant $\Theta >0$ that does not depend on $s$.\end{theorem}
\begin{proof}
As a first step, we want to show that any $s\in\hh\setminus\{0\}$ with $\Re(s) = 0$ belongs to $\rho_S(T)$, i.e. that 
\begin{align*}
\Q_{s}(T) = & T^2 - 2s_0T + |s|^2\id,
\end{align*}
has a bounded inverse, and we want to do this by applying the  Lax-Milgram lemma.
Since $T$ has commuting components, this operator has for $s = \uI s_1 \in\hh$ the form
\begin{align*}
\Q_{s}(T) = T^2 + s_1^2\id =  - (a_1(x_1)\partial_{x_1})^2 - (a_2(x_2)\partial_{x_2})^2 - (a_3(x_3)\partial_{x_3})^2 + s_1^2\mathcal{I}.
\end{align*}
We want to show that it has an inverse on $L^2(\Omega,\hh)$ that satisfies
\[
\left\|\Q_{s}(T)^{-1}\right\|_{\mathcal{B}(L^2)}  \leq C \frac{1}{s_1^2}
\]
with a constant $C$ that is independent of $s_1$. We consider the bilinear form
\[
b(u,v) := \langle \Q_{s}(T)u,v\rangle_{L^2} = \int_{\Omega} \overline{\Q_{s}(T)u(\vx)} v(\vx)\,d\vx
\]
on $H_0^1(\Omega,\hh)$. We rewrite this as
\begin{align*}
b(u,v) & = s_1^2\int_{\Omega}\overline{u(\vx)}v(\vx)\,d\vx - \sum_{\ell = 1}^3 \int_{\Omega}\overline{(a_{\ell}(x_{\ell})\partial_{x_{\ell}})^2u(\vx)}v(\vx)\,d\vx
\end{align*}
and further
\begin{align*}
&\int_{\Omega}\overline{(a_{\ell}(x_{\ell})\partial_{x_{\ell}})^2u(\vx)}v(\vx)\,d\vx \\
=& \int_{\Omega}\left( (a_{\ell}(x_{\ell})\partial_{x_{\ell}})^2\overline{u(\vx)}\right) v(\vx)\,d\vx \\
=& \int_{\Omega}\left(\partial_{x_{\ell}}a_{\ell}(x_{\ell})\partial_{x_{\ell}}\overline{u(\vx)}\right)a_{\ell}(x_{\ell})v(\vx)\,d\vx.
\end{align*}
Integrating by parts, we find
\begin{align*}
&\int_{\Omega}\overline{(a_{\ell}(x_{\ell})\partial_{x_{\ell}})^2u(\vx)}v(\vx)\,d\vx \\
= & -\int_{\Omega}\frac{1}{2}\left(\partial_{x_{\ell}}a_{\ell}(x_{\ell})^2\right) \left(\partial_{x_{\ell}} \overline{u(\vx)}\right)v(\vx)\,d\vx\\
& - \int_{\Omega} a_{\ell}(x_{\ell})^2 \left(\partial_{x_{\ell}} \overline{u(\vx)}\right) \partial_{x_{\ell}} v(\vx)\,d\vx\\
& + \int_{\partial \Omega} a_{\ell}(x_{\ell})^2 \left(\partial_{x_{\ell}} \overline{u(\vx)}\right)v(\vx)n_{\ell}(\vx)\,dS(\vx),
\end{align*}
where $S$ is the surface measure on $\partial\Omega$ and $n_{\ell}(\vx)$ denotes for $\vx\in\partial\Omega$ the $\ell$-th component of the outward pointing normal. Since $v\in H_{0}^{1}(\Omega,\hh)$, the integral over the boundary is zero and hence we altogether obtain
\begin{align*}
b(u,v) = & s_1^2\int_{\Omega}\overline{u(\vx)}v(\vx)\,d\vx + \sum_{\ell = 1}^3 \int_{\Omega}\frac{1}{2}\left(\partial_{x_{\ell}}a_{\ell}(x_{\ell})^2\right)\left(\partial_{x_{\ell}} \overline{u(\vx)}\right)v(\vx)\,d\vx\\
& + \sum_{\ell = 1}^3  \int_{\Omega} a_{\ell}(x_{\ell})^2 \left(\partial_{x_{\ell}} \overline{u(\vx)}\right) \partial_{x_{\ell}} v(\vx)\,d\vx.
\end{align*}

The bilinear form $b$ is continuous on $H_0^1(\Omega,\hh)$ as
\begin{align*}
|b(u,v)| \leq & s_1^2\int_{\Omega}\left| \overline{u(\vx)}v(\vx)\right| \,d\vx + \sum_{\ell = 1}^3 \frac{1}{2}\left\| \partial_{x_{\ell}}a_{\ell}^2\right\|_{\infty}
\int_{\Omega}\left|\left(\partial_{x_{\ell}} \overline{u(\vx)}\right)v(\vx)\right|\,d\vx\\
& + \sum_{\ell = 1}^3  \left\| a_{\ell}^2 \right\|_{\infty}  \int_{\Omega} \left(\partial_{x_{\ell}} \overline{u(\vx)}\right) \partial_{x_{\ell}} v(\vx)\,d\vx \displaybreak[2]\\
\leq & s_1^2 \| u\|_{L^2}\| v\|_{L^2} + \sum_{\ell = 1}^3 \left\| \partial_{x_{\ell}}a_{\ell}^2\right\|_{\infty} \left\|\partial_{x_{\ell}}u\right\|_{L^2} \|v\|_{L^2} \\
&+ \sum_{\ell=1}^3 \left\| a_{\ell}^2\right\|_{\infty} \left\| \partial_{x_{\ell}} u \right\|_{L^2} \left\|\partial_{x_{\ell}} v \right\|_{L^2}\\
\leq & \left( s_1^2 + \sum_{\ell = 1}^3 \left\|\partial_{x_{\ell}} a_{\ell}^2\right\|_{\infty} + \sum_{\ell = 1}^3 \left\|a_{\ell}^2\right\|_{\infty}\right) \|u\|_{H^1}\|v\|_{H^1},
\end{align*}
where $\|\cdot\|_{\infty}$ denotes the supremum norm. Furthermore, observe that for any $w\in L^2$, the map $\ell_{w}(v):= \langle w,v\rangle_{L^2}$ is a continuous quaternionic linear functional on $H_0^1(\Omega,\hh)$ as
\[
|\ell_{w}(v)| = |\langle w,v\rangle_{L^2}| \leq \|w\|_{L^2} \|v\|_{L^2} \leq \|w\|_{L^2} \|v\|_{H_0^1}.
\]

We can consider $H_0^1(\Omega,\hh)$ also as a real Hilbert space, if we restrict the multiplication with scalars to $\rr$ and endow it with the real scalar product $\langle u,v\rangle_{\rr} = \Re\langle u, v\rangle_{H^1}$. Then $\Re\, b$ is a continuous $\rr$-bilinear form on $H_0^1(\Omega,\hh)$  and $\Re \,\ell_{w}$ is for any $w\in L^2(\Omega,\hh)$ a continuous linear functional on $H_0^1(\Omega,\hh)$. What remains to show in order to apply the Lemma of Lax-Milgram is that $\Re\, b$ is even coercive. We have
\begin{align*}
\Re \, b(u,u)  = &s_1^2 \| u\|_{L^2}^2 + \sum_{\ell = 1}^3 \Re \int_{\Omega}\frac{1}{2}\left( \partial_{x_{\ell}}a_{\ell}(x_{\ell})^2\right) \left(\partial_{x_{\ell}} \overline{u(\vx)}\right) u(\vx)\,d\vx\\
& + \sum_{\ell = 1}^3  \int_{\Omega} a_{\ell}(x_{\ell})^2 |\partial_{x_{\ell}} u(\vx))|^2\,d\vx\\
\geq &s_1^2 \| u\|_{L^2}^2 - \sum_{\ell = 1}^3 \frac{1}{2}\left\| \partial_{x_{\ell}}a_{\ell}(x_{\ell})^2\right\|_{\infty}  \int_{\Omega}\left|\partial_{x_{\ell}} \overline{u(\vx)}\right| |u(\vx)|\,d\vx\\
& + \sum_{\ell = 1}^3 \inf_{\vx\in\Omega} \left| a_{\ell}(x_{\ell})^2\right| \int_{\Omega}  |\partial_{x_{\ell}} u(\vx)|^2\,d\vx.
\end{align*}
Applying the Young inequality, we find for any $\delta >0 $ that
\begin{align*}
\Re \, b(u,u)  \geq& s_1^2 \| u\|_{L^2}^2  +  \sum_{\ell = 1}^3 \inf_{\vx\in\Omega} \left| a_{\ell}(x_{\ell})^2\right| \left\|\partial_{x_{\ell}} u\right\|_{L^2}^2\\
&- \sum_{\ell = 1}^3 \frac{1}{2}\left\| \partial_{x_{\ell}}a_{\ell}(x_{\ell})^2\right\|_{\infty}  \left( \frac{\delta}{2} \left\| \partial_{x_{\ell}}u\right\|_{L^2}^2 + \frac{1}{2\delta}\left\| u\right\|_{L^2}^2 \right).
\end{align*}
Therefore we furthermore get
\begin{align*}
\Re\, b(u,u) \geq & \left(s_1^2  - \frac{1}{4\delta} \sum_{\ell = 1}^3 \left\| \partial_{x_{\ell}} a_{\ell}(x_{\ell})^2\right\|_{\infty}\right) \| u\|_{L^2}^2 \\
&+  \sum_{\ell = 1}^3\left(\inf_{\vx\in\Omega} \left| a_{\ell}(x_{\ell})^2\right| - \frac{\delta}{4}\left\| \partial_{x_{\ell}}a_{\ell}(x_{\ell})^2\right\|_{\infty}\right) \left\| \partial_{x_{\ell}} u \right\|_{L^2}^2 \displaybreak[2]\\
\geq  & s_1^2 \| u\|_{L^2}^2    - \frac{C_{\Omega}}{4\delta} \sum_{\ell = 1}^3 \left\| \partial_{x_{\ell}} a_{\ell}(x_{\ell})^2\right\|_{\infty} \left\|\partial_{x_{\ell}}u\right\|_{L^2}^2 \\
&+\sum_{\ell = 1}^3 \left(\inf_{\vx\in\Omega} \left| a_{\ell}(x_{\ell})^2\right| - \frac{\delta}{4}\left\| \partial_{x_{\ell}}a_{\ell}(x_{\ell})^2\right\|_{\infty}\right)  \left\| \partial_{x_{\ell}} u \right\|_{L^2}^2,
\end{align*}
where $C_{\Omega}$ is the Poincaré constant. The optimal choice $\delta = \sqrt{C_{\Omega}}$ finally yields
\begin{align}
\notag &\Re\, b(u,u)\geq \\
\label{CEST}\geq    & s_1^2 \| u\|_{L^2}^2   + \sum_{\ell = 1}^3 \left(\inf_{\vx\in\Omega} \left| a_{\ell}(x_{\ell})^2\right| - \frac{\sqrt{C_{\Omega}}}{2}\left\| \partial_{x_{\ell}}a_{\ell}(x_{\ell})^2\right\|_{\infty}  \right) \left\| \partial_{x_{\ell}} u \right\|_{L^2}^2 \\
\notag \geq &\kappa(s_1^2) \| u\|_{H^{1}}^2
\end{align}
with the constant
\[
\kappa(s_1^2) := \min\left\{ s_1^2, \min_{1\leq \ell\leq 3} \left\{\inf_{\vx\in\Omega} \left| a_{\ell}(x_{\ell})^2\right| - \frac{\sqrt{C_{\Omega}}}{2}\left\| \partial_{x_{\ell}}a_{\ell}(x_{\ell})^2\right\|_{\infty}  \right\} \right\}.
\]

Let now $w\in L^2(\Omega,\hh)$. The Lemma of Lax-Milgram implies due to the arguments above the existence of  a unique $u_w\in H_{0}^{1}(\Omega,\hh)$ such that
\begin{equation}\label{LaxMil}
\Re\, b(u_w,v) = \Re \,\ell_w(v) = \Re \,\langle w,v\rangle_{L^2}\quad\forall v\in H_{0}^{1}(\Omega,\hh)
\end{equation}
and in turn also
\begin{equation}
\langle \Q_{s}(T) u_w,v\rangle_{L^2(\rr,\hh)} = b(u_w,v) = \langle w,v\rangle_{L^2}\quad\forall v\in H_{0}^{1}(\Omega,\hh) ,
\end{equation}
because
\[
b(u_w,v)  = \Re\, b(u_w,v) + \sum_{\ell=1}^3 \left(\Re\,b(u_w,-ve_{\ell}) \right) e_{\ell}
\]
and
\[
\langle  w,v\rangle_{L^2} = \Re\langle w,v\rangle_{L^2} + \sum_{\ell = 1}^3  \left(\Re \langle w, -v e_{\ell} \rangle_{L^2}\right) e_{\ell}.
\]
 Furthermore, we have
\[
\|u_w \|_{L^2} \leq \|u_w \|_{H^1} \leq \frac{1}{\kappa(s_1^2)} \| w\|_{L^2}.
\]
The mapping $S:w\to u_{w}$ is therefore a bounded linear mapping on $L^2(\Omega,\hh)$  and so the operator $\Q_{s}(T)$ has a bounded inverse on $L^2(\Omega,\hh)$ with range in $H_0^1(\Omega,\hh)$. From the estimate \eqref{CEST}, we furthermore conclude
\begin{align*}
s_1^2 \|u_w\|_{L^2}^{2} \leq & \Re\, b(u_w,u_w) \leq  |b(u_w,u_w)| =  \left|\langle w,u_w\rangle_{L^2}\right| \leq \|w\|_{L^2} \|u_w\|_{L^2}.
\end{align*}
Therefore, we have
\begin{align*}
\left\|\Q_{s}(T)^{-1} w\right\|_{L^2} = \left\|u_w\right\|_{L^2}  \leq \frac{1}{s_1^2} \left\| w \right\|_{L^2}
\end{align*}
and so
\begin{equation}\label{QQEST}
\left\| \Q_{s}(T)^{-1}\right\|_{\mathcal{B}(L^2)} \leq \frac{1}{s_1^2}.
\end{equation}

Using this estimate, we can now show that the $S$-resolvent of $T$ decays fast enough along the set of purely imaginary quaternions. For any $v\in H_0^1(\Omega,\hh)$, we have that
\begin{align*}
b(u_w,v) =& \langle  \Q_{s}(T) u_w, v \rangle_{L^2} = \left\langle T^2 u_{w},v\right\rangle_{L^2} + s_1^2 \langle u_w,v\rangle_{L^2}
\end{align*}
The first term can be expressed as
\begin{gather*}
 \langle T^2 u_w, v\rangle_{L^2} = \int_{\Omega}\overline{\left(T^2 u_w\right)(\vx)} v(\vx)\,d\vx \\
= \sum_{\ell = 1}^{3}  \int_{\Omega}   a_{\ell}(x_{\ell})\partial_{x_{\ell}} \overline{(Tu_w)(\vx)} (-e_{\ell}) v(\vx)\,d\vx.
\end{gather*}
Integration by parts yields
\begin{align*}
 \left\langle T^2 u_{w},v\right\rangle_{L^2} =  &  \sum_{\ell = 1}^{3}  \int_{\Omega}   \overline{(Tu_{w})(\vx)}\, e_{\ell} \partial_{x_{\ell}} \left(a_{\ell}(x_{\ell})v(\vx)\right)\,d\vx \\
&+ \sum_{\ell = 1}^{3}  \int_{\partial \Omega}   \overline{(Tu_{w})(\vx)}\, n_{\ell}(\vx)(-e_{\ell}) a_{\ell}(x_{\ell})v(\vx)\,dS(\vx)\\
=&    \sum_{\ell = 1}^{3}  \int_{\Omega}   \overline{(Tu_{w})(\vx)} e_{\ell} \left(\partial_{x_{\ell}} a_{\ell}(x_{\ell})\right)v(\vx)\,d\vx\\
&+     \sum_{\ell = 1}^{3}  \int_{\Omega}   \overline{(Tu_{w})(\vx)} e_{\ell} a_{\ell}(x_{\ell})\partial_{x_{\ell}}v(\vx)\,d\vx\\
=&      \int_{\Omega}   \overline{(Tu_{w})(\vx)}  \left( \sum_{\ell = 1}^{3}e_{\ell}\partial_{x_{\ell}} a_{\ell}(x_{\ell})\right)v(\vx)\,d\vx\\
& +     \int_{\Omega}   \overline{(Tu_{w})(\vx)} Tv(\vx)\,d\vx,
\end{align*}
where the integral over the boundary vanishes as $v(\vx) = 0$ on $\partial \Omega$ because $v\in H_0^1(\Omega,\hh)$. We find that
\begin{align*}
b(u_w,v) = & \int_{\Omega}   \overline{(Tu_w)(\vx)} \Phi(\vx)v(\vx)\,d\vx + \langle Tu_w, Tv\rangle_{L^2} + s_1^2 \langle u_w,v\rangle_{L^2}.
\end{align*}
with
\[
\Phi(\vx) :=  \sum_{\ell = 1}^{3}e_{\ell} \partial_{x_{\ell}} a_{\ell}(x_{\ell}).
\]
Choosing $v = u_w$ yields
\begin{align}\label{LEM}
b(u_w,u_w) =  \int_{\Omega}   \overline{(Tu_w)(\vx)} \Phi(\vx) u_w(\vx)\,d\vx + \| Tu_w\|_{L^2}^2 + s_1^2 \|u_w\|_{L^2}^2.
\end{align}
We hence have
\begin{align*}
|b(u_w,u_w)| \geq& \| Tu_w\|_{L^2}^2 + s_1^2 \|u_w\|_{L^2}^2 -  \int_{\Omega}  \left| \overline{(Tu_w)(\vx)} \Phi(\vx) u_w(\vx)\right|\,d\vx \\
\geq&  \| Tu_w\|_{L^2}^2 + s_1^2 \|u_w\|_{L^2}^2 -  \int_{\Omega}  \left| \overline{(Tu_w)(\vx)} \right| \left\| \Phi\right\|_{\infty}\left| u_w(\vx)\right|\,d\vx \\
\geq&  \frac{1}{2} \left\|Tu_w\right\|_{L^2}^2 + s_1^2 \|u_w\|_{L^2}^2   - \frac{1}{2}\left\| \Phi\right\|_{\infty}^2 \left\| u_w \right\|_{L^2}^2,
\end{align*}
where the last identity follows from the Young inequality. In order to estimate the term $\left\| u_w \right\|_{L^2}^2$, we write
$u_{w}(\vx) = u_{w,0}(\vx) + \sum_{\ell=1}^3 u_{w,\ell}(\vx)e_{\ell}$ with $u_{w,\ell}(\vx)\in\rr$. Then
\begin{align*}
\|u_{w}\|_{L^2}^2 =& \sum_{\ell = 0}^3 \|u_{w,\ell}\|_{L^2}^2 \leq C_{\Omega}^2 \sum_{\ell=0}^3\|u_{w,\ell}\|_{D}^2 \\
=& C_{\Omega}^2 \sum_{\ell=0}^3\sum_{k=1}^3\left\|\partial_{x_k}u_{w,\ell}\right\|_{L^2}^2  \leq  C_{\Omega}^2C_a^2 \sum_{\ell=0}^3\sum_{k=1}^3\left\|a_{k}\partial_{x_k}u_{w,\ell}\right\|_{L^2}^2
\end{align*}
with
\[
C_a := \sup_{\substack{\vx\in\Omega \\ \ell = 1,2,3}} \frac{1}{| a_{\ell}(x_{\ell})|} =  \frac{1}{\inf_{\substack{\vx\in\Omega \\ \ell = 1,2,3}} | a_{\ell}(x_{\ell})|}.
\]
Since $u_{w,\ell}$ is real-valued, we furthermore find that
\begin{align*}
\| Tu_{w,\ell}\|_{L^2}^2 =& \int_{\Omega}\left| \sum_{k=1}^3e_ka_k(\vx) \partial_{x_k}u_{w,{\ell}}(\vx)\right|^2\,d\vx\\
=& \sum_{k=1}^3 \int_{\Omega} \left| a_k(\vx) \partial_{x_k}u_{w,{\ell}}(\vx)\right|^2\,d\vx = \sum_{k=1}^3\left\| a_k\partial_{x_k}u_{w,{\ell}}\right\|_{L^2}^2
\end{align*}
and so
\[
\|u_{w}\|_{L^2}^2 \leq C_{\Omega}^2C_{a}^2\sum_{\ell=0}^3\|Tu_{w,\ell}\|_{L^2}^2.
\]
Altogether, we conclude that
\begin{equation}\label{rumpfi}
\begin{gathered}
\frac{1}{2}\|Tu_{w}\|_{L^2}^2 + s_1^2\|u_{w}\|_{L^2}^2 - \frac{1}{2}\|\Phi\|_{\infty}^2C_{\Omega}^2C_{a}^2\sum_{\ell=0}^3\|Tu_{w,\ell}\|_{L^2}^2\\
\leq |b(u_w,u_w)| = |\langle w,u_w\rangle_{L^2}|\leq \|w\|_{L^2}\|u_w\|_{L^2}.
\end{gathered}
\end{equation}

We observe now that the operator $\Q_{s}(T)$ is a scalar operator and hence maps real-valued functions to real-valued functions so that for $r = 0,\ldots,3$
\begin{gather*}
b(u_{w},u_{w,r}) = \left\langle \Q_{s}(T)u_{w},u_{w,r}\right\rangle_{L^2} \\
= \left\langle \Q_{s}(T)u_{w,0},u_{w,r}\right\rangle_{L^2} +  \sum_{\ell=1}^3(-e_{\ell})\left\langle \Q_{s}(T)u_{w,\ell},u_{w,r}\right\rangle_{L^2}
\end{gather*}
with $\left\langle \Q_{s}(T)u_{w,0},u_{w,r}\right\rangle_{L^2}\in\rr$ for $\ell = 1,2,3$. If $w(x) = w_0(\vx) + \sum_{\ell=1}^3 w_{\ell}(\vx)e_{\ell}$ with $w_{\ell}(\vx)\in\rr$, we hence conclude from
\begin{gather*}
\left\langle \Q_{s}(T)u_{w,0},u_{w,r}\right\rangle_{L^2} +  \sum_{\ell=1}^3(-e_{\ell})\left\langle \Q_{s}(T)u_{w,\ell},u_{w,r}\right\rangle_{L^2}  = b(u_{w},u_{w,r}) \\
= \langle w,u_{w,r}\rangle_{L^2} = \langle w_0,u_{w,r}\rangle_{L^2} + \sum_{\ell=1}^3(-e_{\ell})\langle w_{\ell},u_{w,r}\rangle_{L^2}
\end{gather*}
that for $\ell,r=0,\ldots,4$
\[
b(u_{w,\ell},u_{w,r}) = \langle\Q_{s}(T) u_{w,r},u_{w,r}\rangle_L^2 = \langle w_{\ell},u_{w,r}\rangle_{L^2}
\]
and in particular for $\ell=0,\ldots,4$
\[
b(u_{w,\ell},u_{w,\ell}) = \langle\Q_{s}(T) u_{w,\ell},u_{w,\ell}\rangle_L^2 = \langle w_{\ell},u_{w,\ell}\rangle_{L^2}.
\]
Repeating the above arguments, we find that \eqref{rumpfi} also holds for $u_{w,\ell}$ instead of $u_w$. However, since $u_{w,\ell}$ is real-valued and has hence only one component, this estimate then reads as
\begin{equation}\label{rumpfiL}
\begin{gathered}
\frac{1}{2}\|Tu_{w,\ell}\|_{L^2}^2 + s_1^2\|u_{w,\ell}\|_{L^2}^2 - \frac{1}{2}\|\Phi\|_{\infty}^2C_{\Omega}^2C_{a}^2\|Tu_{w,\ell}\|_L^2
\\
\leq |b(u_{w,\ell},u_{w,\ell})| = |\langle w_{\ell},u_{w,\ell}\rangle_{L^2}|
\\
\leq \|w_{\ell}\|_{L^2}\|u_{w,\ell}\|_{L^2} \leq \|w\|_{L^2}\|u_{w}\|_{L^2}.
\end{gathered}
\end{equation}
If we set
\[
K:=  \frac{1}{2} - \frac{1}{2}\|\Phi\|_{\infty}^2C_{\Omega}^2 C_a^2 >0,
\]
then \eqref{rumpfiL} turns into
\[
K \|Tu_{w,\ell}\|_{L^2}^2 + s_1^2\|u_{w,\ell}\|_{L^2}^2 \leq \|w\|_{L^2}\|u_w\|_{L^2},
\]
which implies in particular
\[
\|Tu_{w,\ell}\|_{L^2}^2\leq \frac{1}{K}\|w\|_{L^2}\|u_w\|_{L^2}.
\]
From \eqref{rumpfi}, we finally conclude that
\begin{equation}\label{rumpfi22}
\begin{gathered}
\frac{1}{2}\|Tu_{w}\|_{L^2}^2 + s_1^2\|u_{w}\|_{L^2}^2 \leq  \frac{1}{2}\|\Phi\|_{\infty}^2C_{\Omega}^2C_{a}^2\sum_{\ell=0}^3\|Tu_{w,\ell} \| +  \|w\|_{L^2}\|u_w\|_{L^2}\\
 \leq \left( 1 + \frac{1}{2}\|\Phi\|_{\infty}^2C_{\Omega}^2C_{a}^2\frac{4}{K}\right) \|w\|_{L^2}\|u_w\|_{L^2}
\end{gathered}
\end{equation}
so that, after setting
\[
\tau := \frac{1}{2}\left( 1 + \frac{1}{2}\|\Phi\|_{\infty}^2C_{\Omega}^2C_{a}^2\frac{4}{K}\right)^{-1} >0,
\]
we have
\[
\tau\|Tu_{w}\|_{L^2}^2 \leq \|w\|_{L^2}\|u_w\|_{L^2}.
\]
Since $u_w = \Q_{s}(T)^{-1}w$ we find because of \eqref{QQEST} that $\|u_w\|_{L^2}\leq \frac{1}{s_1^2}\|w\|_{L^2}$  so that in turn
\[
\tau\|Tu_{w}\|_{L^2}^2 \leq \frac{1}{s_1^2} \|w\|_{L^2}^2.
\]
Hence, we have
\[
\left\| T\Q_{s}(T)^{-1}w\right\|_{L^2} = \|Tu_{w}\|_{L^2} \leq \frac{1}{\sqrt{\tau}s_1} \|w\|_{L^2}.
\]
and so
\[
\left\| T \Q_{s}(T)^{-1}\right\| \leq \frac{1}{\sqrt{\tau}s_1}.
\]
If we set
\[
\Theta := 2\max\left\{1,\frac{1}{\sqrt{\tau}}\right\},
\]
then the above estimate and \eqref{QQEST} yield
\begin{equation*}
\begin{split}
\left\|S_R^{-1}(s,T)\right\| =& \left\|(T - \overline{s}\id)\Q_{s}(T)^{-1}\right\| \\
\leq& \left\|T\Q_{s}(T)^{-1}\right\| + \left\|\overline{s}\Q_{s}(T)^{-1}\right\| \leq\frac{ \Theta}{s_1}
\end{split}
\end{equation*}
and
\begin{equation*}
\begin{split}
\left\|S_L^{-1}(s,T)\right\| =& \left\|T\Q_{s}(T)^{-1} - \Q_{s}(T)^{-1}\overline{s}\right\| \\
\leq& \left\|T\Q_{s}(T)^{-1}\right\| + \left\| \Q_{s}(T)^{-1}\overline{s}\right\| \leq\frac{ \Theta}{s_1}
\end{split}
\end{equation*}
for any $s = \uI s_1\in\hh$.
\end{proof}

Thanks to the previous results we are now in the position to prove our crucial results.
\begin{theorem}\label{MDUE}
Let the coefficients of $T$ be as in Theorem~\ref{Thm1} and let $\alpha\in(0,1)$. For any  $v\in\dom(T)$, the integral
\[
P_{\alpha}(T)v := \frac{1}{2\pi}\int_{-\uI\rr} s^{\alpha-1}\,ds_{\uI}\,S_{R}^{-1}(s,T) Tv.
\]
converges absolutely in $L^2(\Omega,\hh)$.
\end{theorem}
\begin{proof}
The right $S$-resolvent equation implies
\[
S_{R}^{-1}(s,T)Tv = sS_{R}^{-1}(s,T)v - v,\qquad \forall v\in\dom(T)
\]
and so
\begin{align*}
& \frac{1}{2\pi}\int_{-\uI\rr} \left\|s^{\alpha-1}\,ds_{\uI}\,S_{R}^{-1}(s,T) Tv\right\|\\
\leq & \frac{1}{2\pi}\int_{-\infty}^{-1} |t|^{\alpha-1} \left\| S_{R}^{-1}(-\uI t,T) \right\| \left\|Tv\right\| \,dt\\
&+
\frac{1}{2\pi}\int_{-1}^{1} |t|^{\alpha-1} \left\| (-\uI t) S_{R}^{-1}(-\uI t,T)v - v\right\| \,dt\\
&+
\frac{1}{2\pi}\int_{1}^{+\infty} t^{\alpha-1} \left\| S_{R}^{-1}(\uI t,T) \right\| \left\|Tv\right\|\,dt.
\end{align*}
As $\alpha\in(0,1)$, the estimate \eqref{SREST} now yields
\begin{align*}
& \frac{1}{2\pi}\int_{-\uI\rr} \left\|s^{\alpha-1}\,ds_{\uI}\,S_{R}^{-1}(s,T) Tv\right\|\\
\leq & \frac{1}{2\pi}\int_{1}^{+\infty} t^{\alpha-1}\frac{\Theta}{t}\left\|Tv\right\| \,dt +
\frac{1}{2\pi}\int_{-1}^{1} |t|^{\alpha-1} \left(|t| \frac{\Theta}{|t|} + 1\right)\|v\| \,dt\\
&+
\frac{1}{2\pi}\int_{1}^{+\infty} t^{\alpha-1} \frac{\Theta}{t} \left\|Tv\right\|\,dt < +\infty.
\end{align*}
 \end{proof}

 We conclude this section by observing that the above result can be proved also in different function spaces, not only in the Hilbert setting, using different techniques.
 Once that the above result is established, this theory 
  opens the way to the study of the corresponding fractional evolution problem.  This field is now under investigation.


\begin{thebibliography}{10}

\bibitem{adler}
 S. Adler,
  {\em Quaternionic Quantum Mechanics and Quaternionic Quantum Fields},
   Volume 88 of {\em International Series of Monographs on Physics}. Oxford University Press, New York.
1995.

 \bibitem{FUCGEN}
D. Alpay, F. Colombo, J. Gantner, D. P. Kimsey,
 Functions of the infinitesimal generator of a strongly continuous quaternionic group,
{\em Anal. Appl. (Singap.)}, {\bf 15}(2017), pp. 279--311.

 \bibitem{acgs}
D.~{Alpay}, F.~{Colombo}, J.~Gantner, I.~{Sabadini},
 A new resolvent equation for the $S$-functional calculus,
 {\em J. Geom. Anal.}, {\bf 25(3)}(2015), pp. 1939--1968.

\bibitem{ack}
D.~{Alpay}, F.~{Colombo},  D. P. Kimsey,
 The spectral theorem for for quaternionic unbounded normal operators based on the $S$-spectrum,
{\em J. Math. Phys.} {\bf 57}(2016), pp. 023503, 27.



 \bibitem{acks2}
D.~{Alpay}, F.~{Colombo},  D. P. Kimsey, I.~{Sabadini},
{ The spectral theorem for unitary operators based on the $S$-spectrum},
 Milan J. Math., {\bf 84} (2016),  41--61.

\bibitem{Hinfty}
 D. Alpay, F. Colombo, T. Qian, I. Sabadini,
 The $H^\infty$functional calculus based on the S-spectrum for quaternionic operators and for n-tuples of noncommuting operators,
{\em J. Funct. Anal.} {\bf 271}(2016), pp. 1544--1584.



 \bibitem{perturbation}
D.~{Alpay}, F.~{Colombo}, I.~{Sabadini},
Perturbation of the generator of a quaternionic evolution operator,
{\em Anal. Appl. (Singap.)}, {\bf 13(4)}(2015), pp. 347--370.

\bibitem{ACSBOOK}
D. Alpay, F. Colombo, I. Sabadini,
{\em Slice Hyperholomorphic Schur Analysis},
Volume 256 of {\em Operator Theory: Advances and Applications}. Birkh\"{a}user, Basel. 2017.


\bibitem{Balakrishnan}
A. V. Balakrishnan,
 Fractional powers of closed operators and the semigroups generated by them,
{\em Pacific J. Math.}, {\bf 10}(1960), pp. 419--437.

\bibitem{BvN}
G. Birkhoff,  J. von Neumann,
 The logic of quantum mechanics,
   {\em Ann. of Math. (2)}, {\bf 37(4)}(1936), pp. 823--843.

\bibitem{BREZIS}
H. Brezis, {\em Functional analysis, Sobolev spaces and partial differential equations},
 Universitext. Springer, New York, 2011. xiv+599 pp.

\bibitem{BocurValdinoci}
 C. Bucur, E. Valdinoci,
  {\em Nonlocal diffusion and applications},
   Volume 20 of {\em Lecture Notes of the Unione Matematica Italiana},
   Springer, [Cham]; Unione Matematica Italiana, Bologna. 2016.


\bibitem{CafSil}
 L. Caffarelli, L. Silvestre,
   An extension problem related to the fractional Laplacian,
 {\em Comm. Partial Differential Equations}, {\bf 32}(2007), pp. 1245--1260.

  \bibitem{CSV}
 L. Caffarelli,  F. Soria, J. L. Vazquez,
  Regularity of solutions of the fractional porous medium flow,
  {\em J. Eur. Math. Soc. (JEMS)}, {\bf 15}(2013), pp. 1701--1746.

\bibitem{CV2}
L. Caffarelli, J. L. Vazquez,
 Nonlinear porous medium flow with fractional potential pressure,
{\em Arch. Ration. Mech. Anal.}, {\bf 202} (2011), pp. 537--565.

 \bibitem{FJTAMS}
 F. Colombo, J. Gantner,
Fractional  powers of quaternionic operators and Kato's formula using slice hyperholomorphicity,
 Trans. Amer. Math. Soc., {\bf 370} (2018), no. 2, 1045--1100.



\bibitem{64FRAC}
F. Colombo, J. Gantner,
An application of the $S$-functional calculus to fractional diffusion processes,
Preprint 2017.

\bibitem{CGKBOOK} F. Colombo, J. Gantner, D.P. Kimsey,
{\em Spectral theory on the $S$-spectrum for quaternionic operators},
Preprint 2017.



\bibitem{JGA} F. Colombo, I. Sabadini,
 On some properties of the quaternionic functional calculus,
{\em J. Geom. Anal.}, {\bf 19(3)}(2009), pp. 601--627.

\bibitem{CLOSED} F. Colombo, I. Sabadini,
 On the  formulations of the quaternionic functional calculus,
 {\em J. Geom. Phys.}, {\bf 60(10)}(2010), pp. 1490--1508.

\bibitem{Where}
F. Colombo, I. Sabadini,
The $F$-spectrum and the $SC$-functional calculus,
 {\em Proc. Roy. Soc. Edinburgh Sect. A} {\bf 142(3)}(2012), pp. 479--500.

\bibitem{evolution}
F. Colombo, I. Sabadini,
 The quaternionic evolution operator,
{\em  Adv. Math.} {\bf 227(5)}(2011). pp.1772--1805.

 \bibitem{CSSJFA}
F. Colombo, I. Sabadini, D.~C. Struppa,
A new functional calculus for noncommuting operators,
{\em J. Funct. Anal.} {\bf 254}(2008), pp. 2255--2274.



\bibitem{ACSBOOK2}
F. Colombo, I. Sabadini, D.C. Struppa,
{\em Entire Slice Regular Functions}, Volume of {\em
SpringerBriefs in Mathematics}. Springer International Publishing. 2016.

 \bibitem{MR2752913}
F. Colombo, I. Sabadini, D.~C. Struppa,
 {\em Noncommutative functional calculus. Theory and applications of slice hyperholomorphic functions},
 Volume 289 of {\em Progress
  in Mathematics}.
 Birkh\"auser/Springer Basel AG, Basel. 2011.


 \bibitem{MatSpec}
D. R. Farenick, B. A. F. Pidkowich,
 The spectral theorem in quaternions,
 {\em Linear Algebra Appl.} {\bf 371}(2003), pp. 75--102.

\bibitem{DA}
J. Gantner,
A direct approach to the $S$-functional calculus for closed operators,
{\em J. Operator Theory},  {\bf 77(2)}(2017), pp. 101--145.



\bibitem{SpecOP}
J. Gantner,
 Operator Theory on One-Sided Quaternionic Linear Spaces: Intrinsic S-Functional Calculus and Spectral Operators.
To appear in Mem. Amer. Math. Soc..

\bibitem{JONAQS}
J. Gantner,
On the equivalence of complex and quaternionic quantum mechanics,
J. Quantum Stud.: Math. Found. (2017). https://doi.org/10.1007/s40509-017-0147-5

\bibitem{GSSb}
 G. Gentili, C. Stoppato, D. C.  Struppa,
  {\em Regular functions of a quaternionic variable},
Volume of {\em Springer Monographs in Mathematics}. Springer, Heidelberg. 2013.



\bibitem{GR} R. Ghiloni, V. Recupero,
Semigroups over real alternative *-algebras: generation theorems and spherical sectorial operators,
{\em Trans. Amer. Math. Soc.},
 {\bf 368(4)}(2016), pp. 2645--2678.


\bibitem{GMPunzo}
G. Grillo, M. Muratori, F. Punzo,
 Fractional porous media equations: existence and uniqueness of weak solutions with measure data,
{\em Calc. Var. Partial Differential Equations}, {\bf 54}(2015), pp. 3303--3335.

\bibitem{Guzman1}
A. Guzman,
Growth properties of semigroups generated by fractional powers of certain linear operators,
{\em J. Funct. Anal.}, {\bf 23(4)}(1976), pp. 331--352.



\bibitem{Haase}
M. Haase,
  {\em The functional calculus for sectorial operators},
   Volume 169 of {\em Operator Theory: Advances and Applications}. Birkh\"{a}user, Basel. 2006.


\bibitem{Kato}
 T. Kato,
   Note on fractional powers of linear operators, {\em Proc. Japan Acad.} {\bf 36}(1960), pp. 94--96.

\bibitem{Komatsu1}
H. Komatsu,
 Fractional powers of operators, {\em Pacific J. Math.} {\bf 19}(1966), pp. 285--346.



\bibitem{McI1}
A. McIntosh,
  Operators which have an $H^\infty$ functional calculus.
   In: {\em  Miniconference on operator theory and partial differential equations (North Ryde, 1986)}, pp. 210--231. Volume 14 of {\em Proc. Centre Math. Anal. Austral. Nat. Univ.}. Austral. Nat. Univ., Canberra. 1986.


\bibitem{IRE1}
B. Muraleetharan, I.  Sabadini, K. Thirulogasanthar,
S-Spectrum and the quaternionic Cayley transform of an operator,
 J. Geom. Phys., {\bf 124} (2018), 442--455.

\bibitem{IRE2}
B. Muraleetharan, K. Thirulogasanthar, I.  Sabadini,
A representation of Weyl-Heisenberg Lie algebra in the quaternionic setting,
 Ann. Physics, {\bf 385} (2017), 180--213.

\bibitem{Vazquez}
J. L. Vazquez,
{\em The porous medium equation. Mathematical theory.}
Volume of {\em Oxford Mathematical Monographs}.
The Clarendon Press, Oxford University Press, Oxford. 2007.

\bibitem{Watanabe}
J. Watanabe,
  On some properties of fractional powers of linear operators,
{\em Proc. Japan Acad.}, {\bf 37}(1961), pp. 273--275.

\bibitem{Yosida}
K. Yosida,
  Fractional powers of infinitesimal generators and the analyticity of the semi-groups generated by them,
{\em Proc. Japan Acad.}, {\bf 36}(1960). pp. 86--89.






\end{thebibliography}
\end{document}